\documentclass{article}
\usepackage[utf8]{inputenc}

\usepackage{amsmath,amsfonts,amssymb,amsthm,fancyhdr,bm,verbatim,hyperref}
\usepackage{fullpage}
\usepackage{mathtools}
\usepackage{todonotes}
\usepackage[shortlabels]{enumitem}

\makeatletter
\numberwithin{equation}{section}
\numberwithin{figure}{section}
\theoremstyle{plain}
\newtheorem{thm}{\protect\theoremname}
\theoremstyle{plain}

\theoremstyle{plain}
\newtheorem{lem}[thm]{\protect\lemmaname}

\newtheorem{prop}[thm]{\protect\propositionname}
\newtheorem{prob}[thm]{\protect\problemname}
\numberwithin{thm}{section}
\theoremstyle{definition}

\newcommand{\E}{\mathbb{E}}
\newcommand{\Z}{\mathbb{Z}}
\newcommand{\gr}{\textnormal{gr}}
\newcommand{\eps}{\varepsilon}

\newcommand{\jf}[1]{\textcolor{purple}{\textbf{[JF:} #1\textbf{]}}}

\newcommand{\kfm}{K_{4}^{(3)} - e}
\newcommand{\lk}[1]{S_{#1}^{(3)}}
\newcommand{\kn}[2]{K_{#1}^{(#2)}}
\newcommand{\grid}[2]{G_{#1 \times #2}}
\newcommand{\bin}{\textnormal{Bin}}

\usepackage{babel}
\providecommand{\corollaryname}{Corollary}
\providecommand{\lemmaname}{Lemma}
\providecommand{\theoremname}{Theorem}
\providecommand{\conjecturename}{Conjecture}
\providecommand{\propositionname}{Proposition}
\providecommand{\problemname}{Problem}

\title{Hypergraph Ramsey numbers of cliques versus stars}
\author{David Conlon\thanks{Department of Mathematics, California Institute of Technology, Pasadena, CA 91125. Email: dconlon@caltech.edu. Research supported by NSF Award DMS-2054452.} \and
Jacob Fox\thanks{Department of Mathematics, Stanford University, Stanford, CA 94305. Email: jacobfox@stanford.edu. Research supported by a Packard Fellowship and by NSF Awards DMS-1800053 and DMS-2154169.} \and
Xiaoyu He\thanks{Department of Mathematics, Princeton University, Princeton, NJ 08544. Email: xiaoyuh@princeton.edu. Research supported by NSF award DMS-2103154.} \and
Dhruv Mubayi\thanks{Department of Mathematics, Statistics and Computer Science, University of Illinois, Chicago, IL 60607. Email: mubayi@uic.edu. Research partially supported by NSF awards DMS-1763317,
DMS-1952767, DMS-2153576 and a Humboldt Research Award.} \and
Andrew Suk\thanks{Department of Mathematics, University of California at San Diego, La Jolla, CA 92093. Email: asuk@ucsd.edu. Research supported by an NSF
CAREER award and NSF award DMS-1952786.} \and 
Jacques Verstra\"ete\thanks{Department of Mathematics, University of California at San Diego, La Jolla, CA 92093. Email: jacques@ucsd.edu. Research supported by NSF award DMS-1800332.}}
\date{}

\begin{document}

\maketitle

\begin{abstract}
    Let $K_m^{(3)}$ denote the complete $3$-uniform hypergraph on $m$ vertices and $\lk{n}$ the $3$-uniform hypergraph on $n+1$ vertices consisting of all $\binom{n}{2}$ edges incident to a given vertex.  Whereas many hypergraph Ramsey numbers grow either at most polynomially or at least exponentially, we show that the off-diagonal Ramsey number $r(\kn{4}{3},\lk{n})$ exhibits an unusual intermediate growth rate, namely,
    \[
    2^{c \log^2 n} \le r(\kn{4}{3},\lk{n}) \le 2^{c' n^{2/3}\log n}
    \]
    for some positive constants $c$ and $c'$. 
    The proof of these bounds 
    brings in a novel Ramsey problem on grid graphs which may be of independent interest: what is the minimum $N$ such that any $2$-edge-coloring of the Cartesian product $K_N \square K_N$ contains either a red rectangle
    or a blue $K_n$? 
    
\end{abstract}

\section{Introduction}
A {\it $k$-uniform hypergraph} (henceforth, $k$-graph) $G=(V,E)$ consists of a vertex set $V$ and an edge set $E\subseteq \binom{V}{k}$. In particular, we write $\kn{n}{k}$ for the complete $k$-graph with $V = [n]$ and $E = \binom{[n]}{k}$. Ramsey’s theorem states that for any $k$-graphs $H_1$ and $H_2$ there
is a positive integer $N$ such that any $k$-graph $G$ of order $N$ either contains $H_1$ (as a subgraph) or its complement $\overline{G}$ contains $H_2$. The Ramsey number $r(H_1, H_2)$ is the smallest such $N$ and the main objective of hypergraph Ramsey theory is to determine the size of these Ramsey numbers. In particular, for fixed $k\ge 3$, the two central problems in the area are to determine the growth rate of the ``diagonal'' Ramsey number $r(\kn{n}{k},\kn{n}{k})$ as $n\rightarrow \infty$ and the ``off-diagonal'' Ramsey number $r(\kn{m}{k},\kn{n}{k})$ where $m>k$ is fixed and $n\rightarrow \infty$. 

Seminal work of Erd\H{o}s--Rado \cite{ErRa} and Erd\H{o}s--Hajnal (see, e.g., \cite{CoFoSu2}) reduces the estimation of diagonal Ramsey numbers for $k \geq 3$ to the $k=3$ case. For off-diagonal Ramsey numbers, the only case for which the tower height of the growth rate is not known is  $r(\kn{k+1}{k},\kn{n}{k})$, though it was noted in~\cite{MSBLMS} that this tower height could be determined by proving that the 4-uniform Ramsey number $r(\kn{5}{4},\kn{n}{4})$ is double exponential in a power of $n$. Moreover, it was shown in~\cite{MSJCTB} that if $r(\kn{n}{3},\kn{n}{3})$
 grows double exponentially in a power of $n$, then the same is also true for $r(\kn{5}{4},\kn{n}{4})$. Hence, the growth rate for all diagonal and off-diagonal hypergraph Ramsey numbers would follow from knowing the growth rate of the diagonal Ramsey number when $k=3$. Because of this pivotal role, we will restrict our attention below to the case $k=3$.
 
Despite considerable progress in this area in recent years, our state of knowledge about the two central problems mentioned above remains rather dismal.
The best known bounds in the diagonal case (see, e.g., the survey~\cite{CoFoSu3}) are of the form
\begin{equation}\label{eq:diagonal-hypergraph-ramsey}
2^{c n^2} \le r(K_n^{(3)}, K_n^{(3)}) \le 2^{2^{c' n}}
\end{equation}
for some positive constants $c$ and $c'$, differing by an entire exponential order. For the off-diagonal case, when $s \geq 4$ is fixed, the best known bounds \cite{CoFoSu} are of the form
\begin{equation}\label{eq:offdiagonal-hypergraph-ramsey}
2^{c n \log n} \le r(K_s^{(3)}, K_n^{(3)}) \le 2^{c' n^{s-2}\log n}
\end{equation}
for some positive constants $c$ and $c'$, differing by a power of $n$ in the exponent. 

As a possible approach to improving the lower bounds in (\ref{eq:diagonal-hypergraph-ramsey}) and (\ref{eq:offdiagonal-hypergraph-ramsey}), Fox and He~\cite{FoHe} gave new lower-bound constructions for the Ramsey numbers $r(\kn{n}{3}, \lk{t})$, where $\lk{t}$ is the ``star $3$-graph'' on $t+1$ vertices whose edges are all ${t \choose 2}$ triples containing a given vertex. For $n\rightarrow\infty$, they showed that $r(\kn{n}{3}, \lk{n}) \ge 2^{c n^2}$ for some positive constant $c$, giving a new proof of the lower bound in (\ref{eq:diagonal-hypergraph-ramsey}), while, for $t\ge 3$ fixed, they showed that $r(\kn{n}{3},\lk{t}) = 2^{\Theta(n \log n)}$, giving another proof of the lower bound in  (\ref{eq:offdiagonal-hypergraph-ramsey}). In particular, when $t = 3$, this implies that $r(\kn{n}{3},\kfm)= 2^{\Theta(n \log n)}$. One surprising feature of these results is that they give the same bounds that were previously known for clique Ramsey numbers, but with one of the cliques  replaced by the sparser corresponding star. 

Fox and He essentially settled the growth rate of $r(\kn{n}{3}, \lk{t})$ in two cases: when the star is fixed in size and the clique grows; and when the star and clique grow together. The present paper studies the remaining regime: when the clique is fixed in size and the star grows.   First, a standard application of the Lov\'asz Local Lemma yields the following proposition.

\begin{prop} \label{prop:ll}
There are positive constants $c$ and $c'$ such that 
\[
c \frac{n^2}{\log^2 n} \le r(\kfm, \lk{n}) \le c' \frac{n^2}{\log n}.
\]
\end{prop}
Our main result is that if we replace $\kfm$ by $\kn{4}{3}$, then the Ramsey number has an exotic growth rate intermediate between polynomial and exponential. On the other hand, once we move up to $\kn{5}{3}$ the growth rate stabilizes to exponential.

\begin{thm} \label{thm:main}
There are positive constants $c$ and $c'$ such that 
$$2^{c \log^2 n } \le r(\kn{4}{3},\lk{n})\le 2^{c' n^{2/3}\log n}.$$
Moreover, for each $s\ge 5$, there are positive constants $c_s$ and $c'_s$ such that $2^{c_s n} \le r(\kn{s}{3},\lk{n}) \le 2^{c'_s n}$. 
\end{thm}

The intermediate growth rate of $r(\kn{4}{3},\lk{n})$ is in striking contrast with the so-called ``polynomial-to-exponential transition'' conjecture of Erd\H{o}s and Hajnal~\cite{ErHa}, whose exact statement is technical but roughly states that all off-diagonal hypergraph Ramsey numbers against cliques are either at most polynomial or at least exponential. This conjecture was proved to be true infinitely often when $k=3$ by Conlon, Fox and Sudakov~\cite{CoFoSu} and was then settled in the affirmative for all $k\ge 4$ by Mubayi and Razborov~\cite{MuRa}. As a corollary of Theorem~\ref{thm:main}, we see that no such transition can occur for hypergraph Ramsey numbers against stars. 

Since our results for $r(\kfm, \lk{n})$ and 
$r(\kn{s}{3},\lk{n})$ when $s \ge 5$
are straightforward applications of known methods, we will focus our attention in the remainder of the introduction on
how we estimate 
$r(\kn{4}{3},\lk{n})$. The key idea is to reduce to a novel Ramsey problem involving grids, somewhat reminiscent of the grid case of the famous cube lemma (see, e.g., \cite{CoFoLeSu}) developed by Shelah~\cite{Sh} in his proof of primitive-recursive bounds for Hales--Jewett and van der Waerden numbers. To say more, we need some further definitions.

Let the {\it $m\times n$ grid graph} $\grid{m}{n}$ be the Cartesian product $K_m \square K_n$, i.e., the graph whose vertex set is the rectangular grid $[m]\times[n]$ and whose edges are all pairs of distinct vertices sharing exactly one coordinate. If $m\ge a\ge 1$ and $n\ge b\ge 1$, then an {\it $a \times b$ subgrid} in $\grid{m}{n}$ is an induced copy of $\grid{a}{b}$. In particular, we call a $2\times 2$ subgrid, i.e., a set of four vertices $(x,y),(x,y'),(x',y'),(x',y)$ with the four axis-parallel edges between them, a {\it rectangle}. We will be interested in the Ramsey number $\gr(\grid{2}{2}, K_n)$, which is the smallest $N$ such that in any $2$-edge coloring of $\grid{N}{N}$ there is either a monochromatic red rectangle or a monochromatic blue $K_n$. For such Ramsey numbers, we show the following.

\begin{thm}
\label{thm:main-grid}
There are positive constants $c$ and $c'$ such that 
$$2^{c\log^2 n} \le \gr(\grid{2}{2}, K_n)\le 2^{c' n^{2/3}\log n}.
$$
\end{thm}

In passing, we remark that the best known lower bound for the grid case of Shelah's cube lemma, due to Conlon, Fox, Lee and Sudakov~\cite{CoFoLeSu}, is of the form $2^{c(\log r)^{2.5}/\sqrt{\log \log r}}$ for some positive constant $c$, curiously similar to the lower bound in Theorem~\ref{thm:main-grid}, though there $r$ refers to the number of colors. However, despite the similarities, we were not able to find any nontrivial connection between the two problems.

The connection between $r(\kn{4}{3},\lk{n})$ and $\gr(\grid{2}{2}, K_n)$ is that the latter is equivalent to a natural bipartite variant of the former. 
Let $B^{(3)}(a,b)$ denote the complete bipartite $3$-graph on $[a+b]$ whose edges are all triples intersecting both $[a]$ and $[a+1, a+b]$. Observe that the $3$-graph $B^{(3)}(a,b)$ and the grid graph $\grid{a}{b}$ have the same number $a\binom{b}{2} + b\binom{a}{2}$ of edges and we can give an explicit correspondence between their edge sets by sending horizontal edges $(x,y)\sim (x',y)$ in the grid to triples $\{x,x',a+y\} \in E(B^{(3)}(a,b))$ and vertical edges $(x,y)\sim (x,y')$ to triples $\{x,a+y,a+y'\}\in E(B^{(3)}(a,b))$. It is easy to check that rectangles in the square grid $\grid{N}{N}$ correspond to copies of $\kn{4}{3}$ in the corresponding bipartite $3$-graph $B^{(3)}(N,N)$, while $n$-cliques in $\grid{N}{N}$ correspond to copies of $\lk{n}$ in $B^{(3)}(N,N)$. Thus, $\gr(\grid{2}{2}, K_n)$ is exactly equal to the smallest $N$ such that any $2$-edge-coloring of $B^{(3)}(N,N)$ contains either a red $\kn{4}{3}$ or a blue $\lk{n}$.

As $B^{(3)}(N,N)\subseteq \kn{2N}{3}$, it follows immediately that $r(\kn{4}{3},\lk{n}) \le 2\gr(\grid{2}{2}, K_n)$, so the upper bound in Theorem~\ref{thm:main-grid} implies that of Theorem~\ref{thm:main}. We do not know of a direct implication between the lower bounds, but we will be able to glue together copies of our lower-bound construction for $\gr(\grid{2}{2},K_n)$ to manufacture one for $r(\kn{4}{3}, \lk{n})$ of comparable size. This suggests, and we strongly believe, that $r(\kn{4}{3},\lk{n})$ and $\gr(\grid{2}{2}, K_n)$ are of roughly the same order.

The final result that we mention here  is a generalization of Theorem~\ref{thm:main-grid} to larger grids.

\begin{thm}\label{thm:general-grid}
There is a positive constant $c$ and, for all fixed $a\ge b \ge 2$, a positive constant $c' = c'_a$ such that 
\[
2^{c\log^2 n} \le \gr(\grid{a}{b}, K_n) \le 2^{c' n^{1-(2^b-1)^{-1}} \log n}.
\]
\end{thm}

The lower bound is an obvious corollary of the lower bound in Theorem~\ref{thm:main-grid}, while the upper bound involves some extra effort, in particular drawing on recent work of the authors on set-coloring Ramsey numbers~\cite{CoFoHeMuSuVe}. One interesting corollary of this result is that there are positive constants $c$ and $c'$ such that
\[2^{c \log^2 n} \le r(\kn{5}{3} - e, \lk{n}) \le 2^{c' n^{2/3}\log n}.\]
The lower bound here simply follows from the fact that  $r(\kn{5}{3} - e, \lk{n}) \ge  r(\kn{4}{3},\lk{n})$.
For the upper bound, note that, just as $\gr(\grid{2}{2}, K_n)$ is equivalent to a bipartite variant of $r(\kn{4}{3},\lk{n})$, we also have that $\gr(\grid{3}{2}, K_n)$ is equivalent to a bipartite variant of $r(\kn{5}{3} - e, \lk{n})$, so that $r(\kn{5}{3} - e, \lk{n}) \le 2 \gr(\grid{3}{2}, K_n)$, yielding the required upper bound. Together with the fact that $r(\kn{5}{3}, \lk{n}) = 2^{\Theta(n)}$, this gives a more complete picture of the transition window.

Throughout the paper, for the sake of clarity of presentation, we systematically omit floor and ceiling signs whenever they are not essential. Moreover, unless otherwise specified, all logarithms are base $2$.

\section{Basic bounds}

In this short section, we prove Proposition~\ref{prop:ll} and the second part of Theorem~\ref{thm:main}.
We will use the Lov\'asz Local Lemma in the following standard form (see, e.g., \cite[Lemma 5.1.1]{AlSp}).

\begin{lem}[Lov\'asz Local Lemma]\label{lem:LLL}
Let $A_1, A_2, \ldots, A_n$ be events in an arbitrary probability space. A directed graph $D=(V,E)$ on the set of vertices $V=[n]$ is called a dependency graph for the events $A_1,\ldots, A_n$ if for each $i$, $1\le i\le n$, the event $A_i$ is mutually independent of all the events $\{A_j : (i,j)\not \in E\}$. Suppose that $D=(V,E)$ is a dependency digraph for the above events and suppose there are real numbers $x_1,\ldots,x_n$ such that $0\le x_i < 1$ and $\Pr[A_i]\le x_i \prod_{(i,j)\in E} (1-x_j)$ for all $1\le i \le n$. Then $\Pr[\bigwedge_{i=1}^n \overline{A_i}] \ge \prod_{i=1}^n (1-x_i)$. In particular, with positive probability no event $A_i$ holds.
\end{lem}

The proof of Proposition~\ref{prop:ll} is now a direct application of Lemma~\ref{lem:LLL} to a suitable random $3$-graph.

\begin{proof}[Proof of Proposition~\ref{prop:ll}.]
The upper bound follows from the fact that $r(\kfm, \lk{n}) \le r(K_3, K_n) + 1$, by specializing to the link of a single vertex, and then applying the bound $r(K_3, K_n) = O(\frac{n^2}{\log n})$ due to Ajtai, Koml\'os and Szemer\'edi~\cite{AjKoSz}.

For the lower bound, we will use Lemma~\ref{lem:LLL}, assuming throughout that $n$ is sufficiently large. Consider a random hypergraph $\Gamma = G^{(3)}(N, p)$ on $N = 10^{-3}n^2 (\log n)^{-2}$ vertices (which we identify with $[N]$ for convenience) with $p = 4(\log n)/n$. We would like to show that, with positive probability, $\Gamma$ contains no $\kfm$ and its complement $\overline{\Gamma}$ contains no $\lk{n}$. Let $S$ be the collection of all $N{N-1 \choose 3}$ copies of $\kfm$ on $[N]$ and, for $s\in S$, let $A_s$ be the event that $s \subseteq \Gamma$. Let $T$ be the collection of all $N{N-1 \choose n}$ copies of $\lk{n}$ on $[N]$ and, for $t\in T$, let $B_t$ be the event that $t\subseteq \overline{\Gamma}$. The probabilities of these events are $\Pr[A_s] = p^3$ and $\Pr[B_t] = (1-p)^{n \choose 2}$. 

Let $D$ be the digraph whose vertex set is $S \cup T$ and whose edges are pairs $(i,j) \in (S\cup T)^2$ which intersect in at least one edge. Thus, $D$ is a dependency digraph for the events $\{A_s\}_{s\in S} \cup \{B_t\}_{t\in T}$. Each vertex in $S$ is adjacent to at most $9N$ other vertices in $S$ and at most $|T|$ vertices in $T$, while each vertex in $T$ is adjacent to at most ${n \choose 2}3N<2n^2N$ vertices in $S$ and at most $|T|$ vertices in $T$. Let $x = 3p^3$ be the local lemma weight for all the $A_s$ events and $y = 1/|T|$ be the local lemma weight for all the $B_t$ events. With this choice of parameters and using that $n$ is sufficiently large and $|T|=N{N-1 \choose n}<N(eN/n)^n$, we have 
\begin{align*}
    x (1-x)^{9N}(1-y)^{|T|} & \ge p^3, \\
    y (1-x)^{2n^2 N} (1-y)^{|T|} & \ge (1-p)^{n \choose 2},
\end{align*}
so the conditions of Lemma~\ref{lem:LLL} are satisfied. Thus, with positive probability none of the events $A_s$ or $B_t$ hold and we obtain $r(\kfm, \lk{n}) \ge N$, as desired.
\end{proof}

The close connection between $r(\kfm, \lk{n})$ and $r(K_3, K_n)$ suggests that $r(\kfm, \lk{n})=\Theta(\frac{n^2}{\log n})$. It seems likely that a proof of this may be possible through a careful analysis of the $(\kfm)$-free process (see, e.g., \cite{BoMuPi} for results of this type in a similar context). However, we have chosen not to pursue this here.

We now prove the bounds on $r(\kn{s}{3}, \lk{n})$ for $s \ge 5$ stated in Theorem~\ref{thm:main}.

\begin{proof}[Proof of Theorem~\ref{thm:main} for $s\ge 5$.]
The upper bound follows from \cite[Theorem 1.4]{FoHe}, which says that \[
r(\kn{s}{3},\lk{n}) < (2s)^{sn}
\]
for all $s,n \ge 3$. The lower bound construction is as follows. Let $K_N$ be a complete graph on $N$ vertices, labelled $v_1,\ldots, v_N$, and let $\phi$ be a $3$-coloring of the edges of $K_N$ which independently colors each edge by a uniform random element of $\Z/3\Z$. If $\kn{N}{3}$ is a complete $3$-graph on the same vertex set, define a $2$-coloring $\chi$ of the edges of $\kn{N}{3}$ where $\chi(v_i, v_j, v_k)$ is red if $\phi (v_i, v_j)+\phi (v_i, v_k) + \phi (v_j,v_k) \equiv 1\pmod 3$ and blue otherwise.

Suppose, for the sake of contradiction, that a red $\kn{5}{3}$ appears in the coloring $\chi$ at vertices $u_1,\ldots, u_5$. If we sum up $\phi(u_i, u_j)+\phi(u_i,u_k)+\phi(u_j, u_k)$ across all $\binom{5}{3}=10$ triples of these vertices, the sum is $1\pmod 3$, since each triple sums to $1\pmod 3$. On the other hand, each summand $\phi(u_i,u_j)$ appears three times in this sum, so the total sum of all these triples must be $0\pmod 3$. This is a contradiction, so no such coloring can have a red $\kn{5}{3}$.

Next, we show that for $N=(3/2)^{n/2}$ and $n \geq 5$, the probability of a blue $\lk{n}$ appearing in $\chi$ is less than $1$. Indeed, consider any given copy of $\lk{n}$ in $\kn{N}{3}$, with central vertex $u$ and clique $w_1,\ldots, w_n$ in the link of $u$. In order for every edge in this copy of $\lk{n}$ to be blue, every color $\phi(w_i,w_j)$ must satisfy $\phi(w_i,w_j) + \phi(u,w_i) + \phi(u,w_j) \not\equiv 1\pmod 3$. Each such event is independent with probability $2/3$, so we obtain that the probability this particular copy of $\lk{n}$ is blue in $\chi$ is exactly $(2/3)^{\binom{n}{2}}$. As there are $N{N-1 \choose n}$ copies of $\lk{n}$ in $\kn{N}{3}$, we see that the expected number of blue $\lk{n}$ is  $$N{N-1 \choose n}\cdot (2/3)^{\binom{n}{2}} < N(eN/n)^n(2/3)^{n(n-1)/2}=(3e/2n)^n < 1.$$ Hence, there is such a coloring with no blue $\lk{n}$.
\end{proof}

\section{The lower bound}

The key ingredient for our lower bound on $r(\kn{4}{3},\lk{n})$ is the following lemma, which states that it is possible to construct a random subgraph of the grid where each row and column looks like a sparse Erd\H{o}s--R\'enyi random graph, but they are coupled in such a way that there are no rectangles and their edge unions are sparse.

\begin{lem}
\label{lem:random-grid}There exists a positive constant $c$ such that, for all $n$
sufficiently large and $N=2^{c \log^{2} n}$, there is a random
subgraph $H\subseteq \grid{N}{N}$ with the following properties:

\begin{enumerate}[label={(\arabic*)}]
\item For every row $r_{y}$, $H[r_{y}]\sim G(N,n^{-3/4})$, i.e., the marginal
distribution of the induced subgraph $H[r_{y}]$ is $G(N,n^{-3/4})$.
Similarly, for every column $c_{x}$, $H[c_{x}]\sim G(N,n^{-3/4})$.
\item There are no rectangles, i.e., no $x, x', y, y'$ with $(x,y)\sim(x,y')\sim(x',y')\sim(x',y)\sim(x,y)$,
in $H$.
\item The edge union of all the row graphs $H[r_{y}]$ lies in a $G(N,n^{-1/8})$.
Similarly, the edge union of all the column graphs $H[c_{x}]$ lies in a
$G(N,n^{-1/8})$.
\end{enumerate}

\end{lem}
Note that properties (1) and (2) of the random graph $H$ 
are already enough to prove the lower bound in Theorem~\ref{thm:main-grid},
since if we let $H$ determine the set of red edges in $\grid{N}{N}$ and its complement
the blue edges, then (2) shows that there are no red rectangles
and (1) implies that w.h.p.~there are no blue $K_{n}$. We will use the additional property (3) to prove the lower bound on $r(\kn{4}{3},\lk{n})$ in Theorem~\ref{thm:main}.

\begin{proof}[Proof of Lemma \ref{lem:random-grid}.]
We give a construction for $H$ which starts by choosing the column graphs in such a way that every pair of columns is edge-disjoint on some large vertex subset. This then allows us to place many edges between these columns without creating a rectangle.

\vspace{2mm}
\noindent
\textbf{Setting up. }Let $N=2^{c \log^{2} n}$. We
start by picking a family of subsets $\{U_{i}\}_{i\le T}$ of $[N]$
with $T=n^{1/2}(\log n)^{10}$ such that each element $y$ lies in $d(y)$ sets, where $d(y) \in [\frac12 (\log n)^{10}, \frac32(\log n)^{10}]$, and every pair of elements lies in at most $\frac{1}{4}\log n$
 subsets. Such a family exists
by the probabilistic method. To see this, pick the sets $U_{i}$  independently 
such that each $j\in[N]$ lies in $U_{i}$ independently with probability
$n^{-1/2}$. The expected value of $d(y)$ is $(\log n)^{10}$. The multiplicative Chernoff bound for a binomial random variable $X$ implies that $\Pr(|X-\E X| \ge \frac12 \E X) < 2\exp(-\frac{1}{12}\E X)$. Together with a union bound over the $N$ elements $y$, this yields that the probability there exists $y$ with $d(y) \not\in [\frac12 (\log n)^{10}, \frac32(\log n)^{10}]$ is less than $N \cdot 2\exp(-\frac{1}{12}(\log n)^{10}) < \frac12$ for $n$ large. The probability that distinct elements $y,y'$ are both in a given
$U_{i}$  is $n^{-1}$, so the probability they are both in
at least $\frac{1}{4}\log n$ of the $U_{i}$ 
is at most
\[
T^{\frac{1}{4}\log n}\cdot(n^{-1})^{\frac{1}{4}\log n}\le2^{-\frac{1}{16} \log^{2} n}
\]
for $n$ large enough. Hence, for $c$ sufficiently small, the probability that there exists a pair of vertices $y,y'$ that lies in at least 
$\frac14 \log n$ sets is at most ${N \choose 2} 2^{-\frac{1}{16} \log^{2} n} < \frac12$. Therefore, the required family $\{U_{i}\}_{i\le T}$ exists. 

Next we show that there is a collection of bipartitions $P_{i}\sqcup Q_{i}=[N]$, one for each $1\le i \le T$, of the set of columns satisfying the following two properties:

\begin{enumerate}[(a)]
    \item Every pair
of columns $x,x'$ lie on opposite sides of $(\frac{1}{2}+o(1))T$
bipartitions. \label{property:pair}
    \item For every horizontal edge $(x,y)\sim(x',y)$, the number of $i$ for which $y\in U_{i}$ and $x,x'$ lie on opposite sides of the bipartition $P_{i}\sqcup Q_{i}$ is $\Theta((\log n)^{10})$.\label{property:edge}
\end{enumerate}
To see that properties \ref{property:pair} and \ref{property:edge} can be satisfied simultaneously, we show that for a random choice of the bipartitions $P_{i}\sqcup Q_{i}=[N]$, both properties hold with high probability. Indeed, if $D_{x,x'}$ is the set of all $i$ for which $x,x'$ lie on opposite sides of the bipartition $P_{i}\sqcup Q_{i}$, then $|D_{x,x'}| \sim \bin(T, 1/2)$ for all choices of $x$ and $x'$.  By the Chernoff bound,
\[
\Pr[|\bin(T,1/2) - T/2| > \eps T] < e^{-\Omega_\eps(T)},
\]
so even after a union bound over all $\binom{N}{2} = e^{O(\log ^2 n)}$ choices of $x$ and $x'$, we have that w.h.p.~$|D_{x,x'}| = (1/2 + o(1))T$ for all $x,x'$. That is, property~\ref{property:pair} holds w.h.p. To check that property~\ref{property:edge} also holds w.h.p., note that another application of the Chernoff bound shows that if $D_{x,x'}(y)$ is the set of all $i \in D_{x,x'}$ satisfying the additional condition that $y\in U_i$, then $|D_{x,x'}(y)| \sim \bin(d(y),1/2)$ must be tightly concentrated around $d(y)/2 = \Theta((\log n)^{10})$, even after taking a union bound over all choices of $x, x'$ and $y$. We may therefore fix a partition $P_i \sqcup Q_i$ for each $i \in [T]$ such that the collection of such partitions satisfies (a) and (b).

To force property (3), we sample two random graphs $R\sim G(N,n^{-1/8})$
and $C\sim G(N,n^{-1/8})$ in advance; we will make sure that the rows
of $H$ only take edges from $R$ and the columns of $H$ only take
edges from $C$. Finally, for each $i \in [T]$, let  $A_{i}=G(|U_{i}|,1/2)$ be a random graph on vertex set $U_i$ and let $B_{i}=\overline{A_{i}}$ be the edge-complement of $A_{i}$. 

We emphasize here that for each $i \in [T]$, the sets $U_i$ and the pairs $(P_i, Q_i)$ are now fixed. Our goal is to define a probability space (a random subgraph $H \subset \grid{N}{N}$) and, thus, all probabilistic statements that follow are with respect to the product space $\left( \prod A_i \right) \times R \times C$. 

\vspace{2mm}
\noindent
\textbf{The columns. }We first decide the columns of $H$. Let $c_{x}$
be the column indexed $x$ in $\grid{N}{N}$. We define $H_{x}$
to be the (random) graph with vertex set $c_x \cong [N]$ such that $(y,y')$ is an edge of $H_{x}$ 
if and only if, for every $U_i$ containing both $y$ and $y'$, either $x \in P_i$ and $(y,y')\in E(A_i)$ or $x\in Q_i$ and $(y,y')\in E(B_i)$.
 In words, on each column we stipulate
that in each of the subsets $U_{i}$, the induced subgraph $H_{x}[U_{i}]$
is a subgraph of one of the two complementary random graphs $A_{i}$
or $B_{i}$, according to which part of the partition $P_{i}\sqcup Q_{i}$
the $x$-coordinate falls into.

Let $E_x(y, y')$ be the event that a given edge $(y, y')$ appears in the random graph $H_x$. By definition, $E_x(y, y')$ occurs if and only if for every $U_i$ containing both $y$ and $y'$, either $x \in P_i$ and $(y,y')\in E(A_i)$ or $x\in Q_i$ and $(y,y')\in E(B_i)$. We have 
$$\Pr[(x \in P_i \wedge (y,y')\in E(A_i)) \bigvee (x\in Q_i \wedge (y,y')\in E(B_i))] = 1/2.$$
There are at most $\frac{1}{4} \log n$ choices of $i$ for which $U_i$ contains both $y$ and $y'$  and these events are independent over $i$. Thus, $\Pr[E_x(y,y')] \ge 2^{-\frac{1}{4}\log n}=n^{-1/4}$. We observe further that $E_x(y,y')$ depends only on the randomness of the single edge $(y,y')$ in $A_i$ and $B_i$, so, for a fixed $x$, these events are mutually independent as $(y,y')$ varies through the possible edges of $H_x$. Thus, we may choose a random subgraph
$H'_{x}\subseteq H_{x}$ with distribution exactly $G(N,n^{-5/8})$.
Finally, we take $H[c_{x}]=H'_{x}\cap R$, which is a random graph with distribution exactly $G(N,n^{-3/4})$, 
proving properties (1) and (3) for the columns.

\vspace{2mm}
\noindent
\textbf{The rows.} Next, we define the horizontal edges of $H$ by
picking the edges between each pair of columns independently. For
each pair of columns $c_{x}$, $c_{x'}$, recall that $D_{x,x'}$ is the
set of all $i\in[N]$ for which $x,x'$ fall on opposite sides of the 
partition $P_{i}\sqcup Q_{i}$. By our choice of the bipartitions,
we know that $|D_{x,x'}|=(1/2+o(1))T$. For each pair $x,x'$, pick
a uniform random $i_{x,x'}\in D_{x,x'}$. Now, in each row $y$, let
$H_{y}$ be the random graph whose edges are exactly those pairs $(x,x')$
for which $U_{i_{x,x'}}\ni y$. A given edge $(x,x')$ appears in $H_y$ if and only if the random index $i_{x,x'}$ is chosen to be one of the $d(y) \in [\frac12 (\log n)^{10}, \frac32(\log n)^{10}]$ indices $i\in D_{x,x'}$ for which $y \in U_i$, while $i_{x,x'}$ is uniform out of $|D_{x,x'}| = (\frac{1}{2}+o(1))T = (\frac{1}{2}+o(1))n^{1/2}(\log n)^{10}$ choices. Thus, each edge appears in $H_{y}$ with probability $\Theta(n^{-1/2})$. Furthermore, for fixed $y$ these events are mutually independent over all choices of possible $(x,x')$, since the random indices $i_{x,x'}$ are chosen independently. We may therefore find a random subgraph
$H'_{y}\subseteq H_{y}$ with distribution exactly $G(N,n^{-5/8})$.
Finally, we take $H[r_{y}]=H'_{y}\cap C$, which is again a random graph with distribution exactly $G(N,n^{-3/4})$, proving properties (1) and (3) for the rows.

\vspace{2mm}
\noindent
\textbf{Property (2).} Suppose that there is a rectangle in $H$, say $(x,y),(x,y'),(x',y),(x',y')$. By the way we picked the horizontal
edges, this means that, for $i=i_{x,x'}$, we have $y,y'\in U_{i}$
and $x,x'$ fall on opposite sides of the bipartition $P_{i}\sqcup Q_{i}$.
If, say, $x\in P_{i}$ and $x'\in Q_{i}$, then we see that $H[c_{x}][U_{i}]\subseteq A_{i}$
and $H[c_{x'}][U_{i}]\subseteq B_{i}$ are disjoint graphs on the
set $U_{i}$, so at most one of the two vertical edges $(x,y)\sim(x,y')$
and $(x',y)\sim(x',y')$ can lie in $H$. Hence, there are no
rectangles in $H$, as desired.
\end{proof}

Now that Lemma \ref{lem:random-grid} is proved, we layer $\log N$
copies of this bipartite construction on top of each other to obtain the lower bound on $r(\kn{4}{3},\lk{n})$ in
Theorem \ref{thm:main}, namely, $r(\kn{4}{3},\lk{n}) \geq 2^{c \log^2 n}$ for some $c > 0$.

\begin{proof}[Proof of the lower bound on $r(\kn{4}{3},\lk{n})$.]
 For $N$ as in Lemma \ref{lem:random-grid}, draw $t=\log N$ independent
samples $H_{1},\ldots,H_{t}$ from the distribution $H$. Identify the vertices of $\kn{N}{3}$
with $[N]\coloneqq\{0,\ldots,N-1\}$ (we use this convention so that each vertex has at most $t$ bits when written in binary).

Each $H_{\ell}$, $1\le\ell\le t$,
gives rise to a two-edge-coloring $\chi_{\ell}$ of a bipartite subgraph
of $\kn{N}{3}$ as follows. Let $\ell(i_{1},i_{2},i_{3})$ denote
the maximum binary bit on which three distinct $i_{1},i_{2},i_{3}\in[N]$
do not agree. Let $\Gamma_{\ell}$ denote the spanning subgraph of
$\kn{N}{3}$ consisting of all edges with $\ell(i_{1},i_{2},i_{3})=\ell$.
For clarity, we write the vertices of an edge in $\Gamma_{\ell}$
as $\{x,y,y'\}$ if $x$ is the vertex which is $0$ on bit $\ell$
and $y,y'$ are the vertices which are $1$, calling these ``vertical
edges'', and as $\{x,x',y\}$ if $x,x'$ are $0$ on bit
$\ell$ and $y$ is $1$, calling these ``horizontal edges''. Let
$\chi_{\ell}$ denote the coloring of $\Gamma_{\ell}$ for which vertical
edges $\{x,y,y'\}\in\Gamma_{\ell}$ are colored red if and only if $(x,y)\sim(x,y')$
is an edge of $H_{\ell}$ and horizontal edges $\{x,x',y\}\in\Gamma_{\ell}$
are colored red if and only if $(x,y)\sim(x',y)$ is an edge of $H_{\ell}$.

We first claim that the colorings $\chi_{\ell}$ contain no red copies of
$\kn{4}{3}$. Indeed, since $\Gamma_\ell$ is bipartite, a copy of $\kn{4}{3}$ in $\Gamma_\ell$ must lie on four vertices $\{x,x',y,y'\}$ where $x,x'$ are $0$ on bit $\ell$ and $y,y'$ are $1$ on bit $\ell$. This induced subhypergraph is red if and only if the four edges $\{x,y,y'\}$, $\{x',y,y'\}$, $\{x,x',y\}$ and $\{x,x',y'\}$ are all red in $\chi_\ell$. This in turn means that the four edges $(x,y)\sim (x,y')$, $(x',y)\sim (x',y')$, $(x,y)\sim (x',y)$ and $(x,y') \sim (x',y')$ are edges of $H_\ell$, forming a rectangle in $H$ and contradicting property (2) of Lemma~\ref{thm:main-grid}. Thus, no red $\kn{4}{3}$ appears in any of the colorings $\chi_{\ell}$.

Write $N_{\ell}(v)$ for the set of all $u\in[N]$
which disagree with $v$ on bit $\ell$ but agree on all higher bits
and write $L_{\ell}(v)$ for the link of $v$ in $\Gamma_\ell$ restricted to $N_{\ell}(v)$.
By the definition of $\Gamma_{\ell}$, $L_{\ell}(v)$ is a complete graph. Moreover, the coloring $\chi_{\ell}$ induces a coloring
on $L_{\ell}(v)$ for each $v$, which, by property (1) of Lemma~\ref{lem:random-grid},
has red edges distributed as in $G(N,n^{-3/4})$. 

We now build a coloring $\chi$ of $\kn{N}{3}$ out of the colorings
$\chi_{\ell}$ as follows. Note that $\Gamma_{1},\ldots,\Gamma_{t}$
form an edge partition of $\kn{N}{3}$. As a starting point, we
let $\chi'(e)=\chi_{\ell}(e)$ for that $\ell$ such that $e\in\Gamma_{\ell}$.
However, this coloring $\chi'$ may now contain some red $\kn{4}{3}$,
so we modify it as follows. For each red $\kn{4}{3}$ in $\chi'$, say with vertices $\{i_{1},i_{2},i_{3},i_{4}\}$, 
mark the triple of vertices $\{i_{1},i_{2},i_{3}\}$ that has the smallest value $\ell$ of $\ell(i_{1},i_{2},i_{3})$. To see that this
triple is unique, suppose it were not and $\ell(i_{1},i_{2},i_{3})=\ell(i_{1},i_{2},i_{4}) = \ell$.
But then all four vertices agree on all higher bits than $\ell$,
so the $\ell$-values of all four $3$-tuples are at most $\ell$
by definition. Thus, all four $3$-tuples among $\{i_{1},i_{2},i_{3},i_{4}\}$
lie in $\Gamma_{\ell}$ and this is a red $\kn{4}{3}$ in the coloring
$\chi_{\ell}$ of $\Gamma_{\ell}$, which is a contradiction. We also observe that if $\{i_1, i_2, i_3\}$ is marked by a red clique on $\{i, i_1, i_2, i_3\}$, then $\ell(i, i_1, i_2) = \ell(i,i_1, i_3) = \ell(i,i_2,i_3) = \ell'$ for some $\ell'>\ell$. That is, the other three edges in this red clique all belong to the same $\chi_{\ell'}$, so we may say that the edge $\{i_1, i_2, i_3\}$ is {\it marked by level $\ell'$} (note that a single edge can be marked by multiple levels). The
coloring $\chi$ is now defined as follows: the red edges of $\chi$
are exactly the unmarked red edges of $\chi'$.

We claim that $\chi$ is a coloring of $\kn{N}{3}$ that contains
no red $\kn{4}{3}$ and, with positive probability, no $\lk{nt}$.
Since every $\kn{4}{3}$ which is red under $\chi'$ contains a marked triple,
$\chi$ indeed has no red $\kn{4}{3}$. It remains to bound the probability of finding a blue $\lk{nt}$.

Fix a vertex $u\in V(\kn{N}{3})$ and suppose a blue $\lk{nt}$
appears in $\chi$ with $u$ as the central vertex. The sets $N_{1}(u),\ldots,N_{t}(u)$ form a partition
of $V(\kn{N}{3})$, so at least one of these contains at least $n$
of the vertices of our $K_{nt}$. Thus, for some $\ell$, there must be $v_{1},\ldots,v_{n}\in N_{\ell}(u)$
forming a blue $\lk{n}$ with $u$ as the central vertex. By the union bound, it will suffice to show that the probability of such an occurrence is smaller than
$N^{-n-1}$. 

Let $\phi$ be the coloring of the copy of $K_{n}$ formed by the  vertices of this copy of $\lk{n}$ other than $u$, with colors given by
$\phi(v_{i},v_{j})=\chi(u,v_{i},v_{j})$.
By construction, the red edges in $\phi$
correspond to red edges in $L_{\ell}(u)$ in $\chi_{\ell}$ that
are unmarked. We first bound the number of marked edges. For each $\ell' > \ell$, let $M_{\ell'}$ be the graph on $\{v_1,\ldots ,v_n\}$ whose edges are pairs $\{v_i, v_j\}$ for which $\{u, v_i, v_j\}$ is marked by level $\ell'$, as defined previously. In other words, $v_i \sim v_j$ in $M_{\ell'}$ if and only if there exists a fourth vertex $w$ for which $\{w, u, v_i\}$, $\{w, u, v_j\}$ and $\{w, v_i, v_j\}$ are all red in $\chi_{\ell'}$. 

We claim that for each level $\ell'>\ell$, the graph $M_{\ell'}$ is contained inside a copy of $G(n,n^{-1/8})$. Indeed, because of property (3) of Lemma~\ref{lem:random-grid} and the definition of $\chi_{\ell'}$, there exists a random graph $M'_{\ell'}\sim G(n,n^{-1/8})$ such that if $\{w, v_i, v_j\}$ is red for any $w$, then $\{v_i, v_j\}$ is an edge of $M'_{\ell'}$. In particular, if $\{v_i, v_j\}$ is an edge of $M_{\ell'}$, then $\{w,v_i,v_j\}$ is red in $\chi_{\ell'}$ for some $w$ and thus $\{v_i, v_j\}$ is an edge of $M'_{\ell'}$ as well. That is, $M_{\ell'}$ is a spanning subgraph of $M'_{\ell'}$, which has distribution $G(n,n^{-1/8})$. 

Whether an edge is marked by a particular level is independent for each level,
so the edge union of the graphs $M_{\ell'}$ with $\ell' > \ell$ is contained inside the edge union of the independent random graphs $M'_{\ell'}$, which is in turn contained inside a single random graph $M$ with distribution $G(n, tn^{-1/8})$. Let $E$ be the event that $M$ has at least $\frac{1}{2} \binom{n}{2}$ edges. That is, $E$ is the event that the edge count of $M$, distributed like $\bin(\binom{n}{2}, tn^{-1/8})$, is at least $\frac{1}{2} \binom{n}{2}$. By the Chernoff bound,
$\Pr[E] \leq 2^{-\Omega(n^2)}$.

An edge $(v_i, v_j)$ is red in $\phi$ if it does not appear in $M$ and the edge $\{u,v_{i},v_{j}\}$ is red in $\chi'_\ell$. The latter occurs with probability $n^{-3/4}$, by property (1) of Lemma~\ref{lem:random-grid}. Thus,
\[
\Pr[\phi(v_{i},v_{j})\text{ red} | (v_i,v_j)\not \in E(M)]\ge n^{-3/4}.
\]
Given any particular choice of $M$, 
all such events are mutually independent, so the probability that $\phi$ is monochromatic blue
is at most
\[
\Pr[E] + \Pr[\phi \textnormal{ is monochromatic blue} | \overline{E}] \le 2^{-\Omega(n^2)} + \left(1-n^{-3/4}\right)^{\frac{1}{2}\binom{n}{2}}\le2^{-\Omega(n^{5/4})},
\]
which suffices to union bound over all $N^{n+1}$ choices of $u,v_{1},\ldots,v_{n}$,
as desired. This completes the proof.
\end{proof}

\section{The upper bound}

In this section, we first prove the upper bound in Theorem~\ref{thm:main-grid}, which states that 
\begin{equation}\label{eq:grid-upper}
\gr(\grid{2}{2}, K_n) \le 2^{c' n^{2/3}\log n}
\end{equation}
for some positive constant $c'$. 
As observed in the introduction, $r(\kn{4}{3},\lk{n}) \le 2\gr(\grid{2}{2}, K_n)$, so this immediately implies the upper bound in Theorem~\ref{thm:main} as well.

The main technical tool used is the following Ramsey-type result of Erd\H{o}s and Szemer\'edi.

\begin{lem}[Erd\H{o}s--Szemer\'edi~\cite{ErSz}]\label{lem:erdos-szemeredi}
There exists a positive constant $c_0$ such that if the edges of a complete graph $K_N$ are colored in $r$ colors, then there exists a clique of order $n = c_0 \frac{r}{\log r} \log N$ and a color $i$ that does not appear on any edge in that clique.
\end{lem}

We are now ready to prove (\ref{eq:grid-upper}).

\begin{proof}[Proof of (\ref{eq:grid-upper})]
Let $c'=\max(2,1/c_0)$, where $c_0$ is the constant in Lemma \ref{lem:erdos-szemeredi}. Let $N= 2^{c' n^{2/3}\log n}$. We would like to show that in any $2$-edge-coloring of $\grid{N}{N}$, there is either a red rectangle or a blue $K_n$. Letting $r= n^{1/3}$, we will in fact prove the stronger statement that the same result holds for the rectangular grid $\grid{N}{M}$, where the height is chosen to be the Ramsey number $M = r(K_r, K_n) \le n^{r} \le N$. 

Fix a $2$-edge-coloring of $\grid{N}{M}$. Each column is a $2$-edge-colored $K_M$, so, by the definition of $M$, each column contains either a red $K_r$ or a blue $K_n$. In the latter case we are already done, so we may assume that every column of the grid contains a red $K_r$. Associate with each column $x$ the $y$-coordinates $\vec y(x) = (y_1,\ldots,  y_r)$ of the vertices of some red $K_r$ in that column. Since $y(x)$ can take at most $M^r$ possible values, there exist $N' = N/M^r$ columns with the same value $\vec y = \vec y(x)$. Since $M \le n^{r}$, we have $M^r \le n^{r^2} = 2^{n^{2/3} \log n}$. As $c' \geq 2$, we have $N' \ge \sqrt{N}$.

Restrict to the $N' \times r$ subgrid where the rows are the $r$ rows indexed by the coordinates of $\vec y$ and the columns are the $N'$ columns $x$ with $\vec y(x)=\vec y$. In this subgrid, every column is a monochromatic red $K_r$. If any pair of columns has at least two red edges between them, then we have a red rectangle and we are done. Thus, there is an edge-coloring 
$h:E(K_{N'}) \rightarrow [r]$ of the complete graph on $[N']$ such that the horizontal edge $(x,y)\sim (x',y)$ is blue whenever $y \ne h(x,x')$.

By Lemma~\ref{lem:erdos-szemeredi}, there is a clique of order
\[
c_0 \frac{r}{\log r} \log N' \ge c_0 \frac{3n^{1/3}}{\log n} \cdot \frac{1}{2} \log N \ge  n
\]
and a color $y$ such that $y  \ne h(x,x')$ for every pair of vertices $x,x'$ in the clique. If the vertices of the clique are $\{x_1,\ldots ,x_n\}$, then all edges between the vertices $\{(x_1,y),\ldots ,(x_n,y)\}$ in the original grid are blue, forming the desired blue $K_n$. This completes the proof.
\end{proof}

Next, we generalize this upper bound to arbitrary grids. Recall that $\grid{a}{b}$ is an $a\times b$ grid graph and $\gr(\grid{a}{b}, K_n)$ is the smallest $N$ such that in any $2$-edge-coloring of $\grid{N}{N}$ there is either a monochromatic red copy of $\grid{a}{b}$ or a monochromatic blue copy of $K_n$. We next prove the upper bound in Theorem \ref{thm:general-grid}, that, for all $a\ge b \ge 2$, there is a positive constant $c' = c'_a$ such that
\begin{equation} \label{eq:gridab-upper}
\gr(\grid{a}{b}, K_n) \le 2^{c' n^{1-(2^b-1)^{-1}} \log n}.
\end{equation}

The proof is essentially the result of iterating the argument for Theorem~\ref{thm:main-grid}. However, we will also require a generalization of Lemma~\ref{lem:erdos-szemeredi} from a recent paper of the authors~\cite{CoFoHeMuSuVe}. Define the {\it set-coloring Ramsey number} $R(n;r,s)$ 
to be the smallest positive integer $N$ such that if every edge of $K_N$ receives a set of $s$ colors from a palette of $r$ colors, then there must exist a copy of $K_n$ where a single common color appears on every edge.

\begin{lem}[{Corollary of \cite[Theorem 1.1]{CoFoHeMuSuVe}}]\label{lem:general-ES}
There is a constant $C_0$ such that the following holds. For any integers $n\ge 3$ and $r > s \ge r/2 \ge 1$,
\[
R(n;r,s) \le 2^{C_0n(r-s)^2 r^{-1} \log (r/(r-s))}.
\]
\end{lem}

Combining the above lemma with a supersaturation argument, we obtain the following result, needed for the iterative step in our proof of \eqref{eq:gridab-upper}. 

\begin{lem}\label{lem:general-grid-iteration} There is a constant $C$ such that the following holds. 
Suppose $r,a,n,N,N'$ are positive integers satisfying $r\ge 2a$, $n \geq (r/a)^2$ and $N\ge 2^{Cn a^2 r^{-1}\log (r/a)} N'$. If the vertical edges of $G: = \grid{N}{r}$ are colored red and the horizontal edges of $G$ are colored red or blue, then $G$ contains a blue $K_n$ or a copy of $\grid{N'}{a}$ where all vertices in some column are only incident to red edges. 
\end{lem}

\begin{proof}
Suppose that $G$ contains no blue $K_n$.
Let $T = 2^{C_0 n a^2 r^{-1}\log (r/a)}$, so that, by Lemma~\ref{lem:general-ES}, $T \geq R(n; r, r-a)$. 
Let $C=\max\{1, 3C_0\}$. 
We claim that any $T \times r$ subgrid of $G$ contains a red $\grid{2}{a}$. Indeed, let $G'$ be a $T \times r$ subgrid of $G$ and define an edge-coloring $\chi$ on the complete graph $K_T$ whose vertices are the columns of $G'$, which colors each edge $(x,x')$ by the set of all $y$ for which the edge $(x,y)\sim (x',y)$ is blue. If there is no copy of a red $\grid{2}{a}$ in $G'$, the edge-coloring $\chi$ assigns at least $r-a$ colors to every edge, so, by the definition of $T$, we obtain a monochromatic $K_n$ in some color in $\chi$. This implies that we have a monochromatic blue $K_n$ in the original graph $G$, a contradiction.  

We now run the supersaturation argument. Define another auxiliary graph $H$ whose vertices are the columns of $G$ and edges are pairs of columns containing a red $\grid{2}{a}$. Color each edge $(x,x')$ of $H$ by a set of $a$ $y$-coordinates $y_1,\ldots, y_a$ such that $G[\{x,x'\}\times \{y_1,\ldots ,y_a\}]$ forms a monochromatic red $\grid{2}{a}$. 
We know that among every $T$ vertices of $H$, there is at least one edge. Hence, by supersaturation, there are at least 
\[\binom{N}{T}/\binom{N-2}{T-2} = \binom{N}{2}/\binom{T}{2} \geq N^2/T^2\]
edges in $H$, at least $N^2/T^2 \binom{r}{a}$ of which receive the same color. Hence, there must be a vertex with at least $N/T^2\binom{r}{a} \geq N/T^3 \geq N'$ neighbors in a single color. But this corresponds exactly to a copy of $\grid{N'}{a}$ with a column incident to only red edges, completing the proof.
\end{proof}

It remains to iterate the above lemma to obtain \eqref{eq:gridab-upper}.

\begin{proof}[Proof of (\ref{eq:gridab-upper})]
To begin, let $a = r_1 < r_2 <  \cdots < r_b$ and $1  = N_1 < N_2 < \cdots < N_b$ be positive integers satisfying $2\le r_{i+1}/r_i \leq \sqrt{n}$ and
\[
N_{i+1} = 2^{Cn r_i ^2 r_{i+1}^{-1} \log (r_{i+1} / r_i)} (N_i+1)
\]
for every $1\le i \le b-1$, where $C$ is the constant from Lemma~\ref{lem:general-grid-iteration}. Let $N \coloneqq n^{r_b^2} \cdot N_b$. We claim that $\gr(\grid{b}{a}, K_n) \le N$.

Indeed, suppose we are given a red-blue edge-coloring of $\grid{N}{N}$. We may further suppose that there is no blue $K_n$. Since $r(K_{r_b}, K_n) \le n^{r_b}$, in each column, there are at least

\[\binom{N}{n^{r_b}}/\binom{N-r_b}{n^{r_b} - r_b} = \binom{N}{r_b}/\binom{n^{r_b}}{r_b}\]

\noindent copies of $K_{r_b}$ that are monochromatic red.  Since

\[N/\binom{n^{r_b}}{r_b} \geq N/n^{r_b^2} = N_b,\]

\noindent by the pigeonhole principle, we obtain an induced subgrid of dimensions $N_b \times r_b$ with monochromatic red columns. Applying Lemma~\ref{lem:general-grid-iteration}, we find inside a subgrid of dimensions $(N_{b-1}+1) \times r_{b-1}$ where one column is complete in red to the others. Iterating this process $b-1$ times and setting the distinguished column aside at each step, we find a monochromatic $\grid{b}{a}$, as desired.

In order to obtain (\ref{eq:gridab-upper}), for $1\leq i \leq b$, we choose $r_i = a \cdot x^{2^{i-1}-1}$ with $x= n^{(2^b-1)^{-1}}$. One can check that these choices imply $2 \le r_{i+1}/r_i \le \sqrt{n}$ and $N_{i+1}=2^{O_a(nx^{-1}\log n)}N_i$ for all $1 \leq i \leq b-1$. Consequently, $N_b=2^{O_a(nx^{-1}\log n)}$ and, as $r_b^2=a^2n/x$, we also have $N=2^{O_a(nx^{-1}\log n)}$. Since, by symmetry, $\gr(\grid{a}{b}, K_n) = \gr(\grid{b}{a}, K_n)$, this yields the required bound.
\end{proof}

\section{Concluding Remarks}

The main problem left open by this paper is what the true bounds are for $r(\kn{4}{3},\lk{n})$ and the closely related function $\gr(\grid{2}{2}, K_n)$. In particular, we have the following question.

\begin{prob}
Does there exist $c > 0$ such that $\gr(\grid{2}{2}, K_n) \geq 2^{n^c}$?
\end{prob}

We will not hazard a guess on which direction the truth should lie, though it would be much more interesting were the answer to turn out negative.

Theorem~\ref{thm:general-grid} gives a sub-exponential bound for grid Ramsey numbers of the form 
\begin{equation}\label{eq:grid-upper-conc}
\gr(\grid{a}{b},  K_n) \le 2^{c' n^{1-(2^b-1)^{-1}} \log n}
\end{equation}
when $a\ge b \ge 2$. The dependence of the exponent of $n$ on $b$, which comes from iterating Lemma~\ref{lem:general-ES}, is inverse exponential. The authors suggested in \cite[Problem 6.1]{CoFoHeMuSuVe} that stronger bounds than Lemma~\ref{lem:general-ES} might be true for the set-coloring Ramsey number $R(n;r,s)$, especially in the regime $s \approx r - \sqrt{r}$ where we are applying it here. Such improved upper bounds on $R(n;r,s)$ would immediately improve the dependence on $b$ in (\ref{eq:grid-upper-conc}).

\vspace{3mm}
\noindent
{\bf Acknowledgements.} This research was initiated during a visit to the American Institute of Mathematics under their SQuaREs program. We would like to thank Noga Alon, Matija Buci\'c, Benjamin Gunby and Yuval Wigderson for helpful conversations.

\end{document}